\documentclass[12pt]{article}
\pdfoutput=1

\usepackage{a4wide}
\usepackage{amsmath}
\usepackage{amssymb}
\usepackage{amsthm}
\usepackage{picins}
\usepackage{graphicx}
\usepackage{wrapfig}
\usepackage{subfigure}
\usepackage{color}
\usepackage[small,bf]{caption}
\usepackage{ifpdf}
\ifpdf
\fi

\newtheorem{theorem}{Theorem}

\newtheorem{corollary}{Corollary}

\theoremstyle{definition}
\newtheorem{example}{Example}

\begin{document}

\begin{center}
{\Huge Boundary length of reconstructions in discrete tomography}

\bigskip

{\Large Birgit van Dalen}

\textit{Mathematisch Instituut, Universiteit Leiden, Niels Bohrweg 1, 2333 CA Leiden, The Netherlands \\ dalen@math.leidenuniv.nl}

\today

\end{center}

{\small \textbf{Abstract:} We consider possible reconstructions of a binary image of which the row and column sums are given. For any reconstruction we can define the length of the boundary of the image. In this paper we prove a new lower bound on the length of this boundary. In contrast to simple bounds that have been derived previously, in this new lower bound the information of both row and column sums is combined.}

\section{Introduction}\label{introduction}

An important problem in discrete tomography is to reconstruct a binary image on a lattice from given projections in lattice directions \cite{boek, boeknieuw}. Each point of a binary image
has a value equal to zero or one. The line sum of a line through the image is the sum of the
values of the points on this line. The projection of the image in a certain direction
consists of all the line sums of the lines through the image in this direction. Any binary image with exactly the same projections as the original image we call a \emph{reconstruction} of the image.

For any set of more than two directions, the problem of reconstructing a binary image from its projections in those directions is NP-complete \cite{gardner}. For exactly two directions, the horizontal and vertical ones, say, it is possible to reconstruct an image in polynomial time. Already in 1957, Ryser described an algorithm to do so \cite{ryser}. He also characterised the set of projections that correspond to a unique binary image.

If there are multiple images corresponding to one set of line sums, it is interesting to reconstruct an image with a special property. In order to find reconstructions that look rather like a real object, two special properties in particular are often imposed on the reconstructions. The first is \emph{connectivity} of the points with value one in the picture \cite{hv-convex2,hv-convex3,woeginger}. The second is \emph{hv-convexity}: if in each row and each column, the points with value one form one connected block, the image is called \emph{hv-convex}. The reconstruction of hv-convex images, either connected or not necessarily connected, has been studied extensively \cite{hv-convex1, hv-convex2, hv-convex3, dahlflatberg, woeginger}.

Another relevant concept in this context is the \emph{boundary} of a binary image. The boundary can be defined as the pairs consisting of two adjacent points, one with value 0 and one with value 1. Here we use 4-adjacency: that is, a point is adjacent to its two vertical and to its two horizontal neighbours \cite{connectivity}. The number of such pairs of adjacent points with two different values is called the \emph{length of the boundary} or sometimes the \emph{perimeter length} \cite{gray}.

In this paper we will consider given line sums that may correspond to more than one binary image. Since the boundary of real objects is often small compared to the area, it makes sense to look for reconstructions of which the length of the boundary is as small as possible. In particular, if there exists an hv-convex reconstruction, then the length of the boundary of that image is the smallest possible. In that sense, the length of the boundary is a more general concept than hv-convexity.

The question we are interested in in this paper is: given line sums, what is the smallest length of the boundary that a reconstruction fitting those line sums can have? We can give two straightforward lower bounds on the length of the boundary, given the row and column sums. Both are equivalent to bounds given by Dahl and Flatberg in \cite[Section 2]{dahlflatberg}.

The first is that every column with a non-zero sum contributes 2 to the length of the horizontal boundary, while every row with non-zero sum contributes 2 to the length of the vertical boundary. So if there are $m$ non-zero row sums and $n$ non-zero column sums, then the total length of the boundary is at least $2n+2m$.

For the second bound we use that if the row sums of two consecutive rows are different, the length of the horizontal boundary between those rows is at least the absolute difference between those row sums. A similar result holds for the column sums and the vertical boundary. So if an image has row sums $r_1$, $r_2$, \ldots, $r_m$ and column sums $c_1$, $c_2$, \ldots, $c_n$, then the length of the boundary is at least
\[
r_1 + \sum_{i=1}^{m-1} |r_i - r_{i+1}| + r_m + c_1 + \sum_{j=1}^{n-1} |c_j - c_{j+1}| + c_n.
\]

Despite being simple, these bounds are sharp in many cases. For example, the first bound is sharp if and only if there exists a hv-convex image that satisfies the line sums. On the other hand it is clear that much information is disregarded in these bounds. The first bound does not use the actual value of the non-zero line sums at all, while the second bound only uses the column sums to estimate the length of the vertical boundary and only the row sums to estimate the length of the horizontal boundary.

In this paper we prove a new lower bound on the length of the boundary that combines the row and column sums. After introducing some notation in Section \ref{notation}, we prove this bound in Section \ref{main}. Some examples and a corollary are in Section \ref{examples}. Finally, in Section \ref{variation} we derive an extension of the bound that gives better results in certain cases.

\section{Definitions and notation}\label{notation}

Let $F$ be a finite subset of $\mathbb{Z}^2$ with characteristic function $\chi$. (That is, $\chi(x,y) = 1$ if $(x,y) \in F$ and $\chi(x,y) = 0$ otherwise.) For $i \in \mathbb{Z}$, we define \emph{row} $i$ as the set $\{(x,y) \in \mathbb{Z}^2: x = i\}$. We call $i$ the index of the row. For $j \in \mathbb{Z}$, we define \emph{column} $j$ as the set $\{(x,y) \in \mathbb{Z}^2: y = j\}$. We call $j$ the index of the column. Following matrix notation, we use row numbers that increase when going downwards and column numbers that increase when going to the right.

The \emph{row sum} $r_i$ is the number of elements of $F$ in row $i$, that is $r_i = \sum_{j \in \mathbb{Z}} \chi(i,j)$. The \emph{column sum} $c_j$ of $F$ is the number of elements of $F$ in column $j$, that is $c_j = \sum_{i \in \mathbb{Z}} \chi(i,j)$. We refer to both row and column sums as the \emph{line sums} of $F$. We will usually only consider finite sequences $\mathcal{R} = (r_1, r_2, \ldots, r_m)$ and $\mathcal{C} = (c_1, c_2, \ldots, c_n)$ of row and column sums that contain all the nonzero line sums.

Given sequences of integers $\mathcal{R} = (r_1, r_2, \ldots, r_m)$ and $\mathcal{C} = (c_1, c_2, \ldots, c_n)$, we say that $(\mathcal{R}, \mathcal{C})$ is \emph{consistent} if there exists a set $F$ with row sums $\mathcal{R}$ and column sums $\mathcal{C}$. Define $b_i = \#\{j: c_j \geq i\}$ for $i = 1, 2, \ldots, m$. Ryser's theorem \cite{ryser} states that if $r_1 \geq r_2 \geq \ldots \geq r_m$, the line sums $(\mathcal{R}, \mathcal{C})$ are consistent if and only if for each $k = 1, 2, \ldots, m$ we have $\sum_{i=1}^k b_i \geq \sum_{i=1}^k r_i$. From this we can conclude a similar result for the case of not necessarily non-increasing row sums: if the line sums $(\mathcal{R}, \mathcal{C})$ are consistent, then for all $k = 1, 2, \ldots, m$ we have
\begin{equation}\label{ryserconsequence}
\sum_{i=1}^k b_i \geq \sum_{i=1}^k r_i.
\end{equation}
The converse clearly does not hold.

We can view the set $F$ as a picture consisting of cells with zeroes and ones. Rather than $(i,j) \in F$, we might say that $(i,j)$ has value 1 or that there is a one at $(i,j)$. Similarly, for $(i,j) \not\in F$ we sometimes say that $(i,j)$ has value zero or that there is a zero at $(i,j)$.

We define the \emph{boundary} of $F$ as the set consisting of all pairs of points $\big( (i,j), (i',j') \big)$ such that
\begin{itemize}
\item $i=i'$ and $|j-j'| =1$, or $|i-i'| = 1$ and $j=j'$, and
\item $(i,j) \in F$ and $(i',j') \not\in F$.
\end{itemize}
One element of this set we call \emph{one piece of the boundary}. We can partition the boundary into two subsets, one containing the pairs of points with $i=i'$ and the other containing the pairs of points with $j=j'$. The former set we call the \emph{vertical boundary} and the latter set we call the \emph{horizontal boundary}. We define the \emph{length of the (horizontal, vertical) boundary} as the number of elements in the (horizontal, vertical) boundary.

\section{The main theorem}\label{main}

\begin{theorem}\label{altgrens}
Let row sums $\mathcal{R} = (r_1, r_2, \ldots, r_m)$ and column sums $\mathcal{C} = (c_1, c_2, \ldots, c_n)$ be given, where $r_1=n$, $r_m = 0$. Let $L_h$ be the total length of the horizontal boundary of an image with line sums $(\mathcal{R}, \mathcal{C})$. Define $b_i = \#\{j: c_j \geq i\}$ and $d_i = b_i - r_i$ for $i = 1, 2, \ldots, m$. For any integer $t \geq 0$ and any subset $\{i_1, i_2, \ldots, i_{2t+1} \} \subset \{1, 2, \ldots, m\}$ with $i_1 < i_2 < \ldots < i_{2t+1}$ we have
\begin{align}
L_h &\geq 2n + d_{i_1} - d_{i_2} + d_{i_3} - \cdots - d_{i_{2t}} + 2d_{i_{2t+1}}, \label{altgrens1} \\
L_h &\geq 2n - d_{i_{2t+1}} + d_{i_{2t}} - d_{i_{2t-1}} + \cdots + d_{i_2} - 2d_{i_1} \label{altgrens2}.
\end{align}
\end{theorem}

\begin{proof}
First we prove (\ref{altgrens1}) by induction on $n$. In the initial case $n=0$ we have $d_i = b_i = r_i = 0$ for all $i$, hence we have to prove that $L_h \geq 0$, which is obviously true.

Now let $n \geq 1$ and consider a binary image $F$ with line sums $(\mathcal{R}, \mathcal{C})$. Let $I \subset \{1, 2, \ldots, m\}$ be the set of indices $i$ such that cell $(i,n)$ has value 1. Note that $\# I = c_n$. Let $F'$ be the binary image we obtain by deleting column $n$ from $F$. Let $(r_1', r_2', \ldots, r_m')$ be the row sums of $F'$. The column sums of $F'$ are $(c_1, c_2, \ldots, c_{n-1})$, and define $b_i' = \#\{ j \leq n-1 : c_j \geq i\}$ and $d_i' = b_i' - r_i'$ for $i=1,2,\ldots, m$. We have
\[
r_i' = \begin{cases} r_i & \text{if } i \not\in I, \\ r_i-1 & \text{if } i \in I, \end{cases}
\]
\[
b_i' = \begin{cases} b_i -1 & \text{if } i \leq c_n, \\ b_i & \text{if } i > c_n, \end{cases}
\]
and therefore
\[
d_i' = \begin{cases} d_i - 1 & \text{if } i \not\in I \text{ and } i \leq c_n, \\
d_i & \text{if } i \notin I \text{ and } i > c_n, \text{ or } i \in I \text{ and } i \leq c_n, \\
d_i + 1 & \text{if } i \in I \text{ and } i > c_n. \end{cases}
\]
As induction hypothesis we assume that (\ref{altgrens1}) is true for the smaller image $F'$. So for the total length $L_h'$ of the horizontal boundary of $F'$ we have
\[
L_h' \geq 2(n-1) + d_{i_1}' - d_{i_2}' + d_{i_3}' - \cdots - d_{i_{2t}}' + 2d_{i_{2t+1}}'.
\]
Let $2B$ be equal to the horizontal boundary in column $n$ of $F$. Then $L_h = L_h' + 2B$. We want to prove (\ref{altgrens1}), hence it suffices to prove
\begin{equation}\label{altgrenstoprove}
2B - 2 \geq (d_{i_1} - d_{i_1}') - (d_{i_2} - d_{i_2}') + (d_{i_3} - d_{i_3}') - \cdots - (d_{i_{2t}} - d_{i_{2t}}') + 2(d_{i_{2t+1}} - d_{i_{2t+1}}').
\end{equation}
Write the right-hand side as
\[
\sum_{s=1}^t \left( (d_{i_{2s-1}} - d_{i_{2s-1}}') - (d_{i_{2s}} - d_{i_{2s}}') \right) + 2(d_{i_{2t+1}} - d_{i_{2t+1}}').
\]
Note that
\[
d_i - d_i' = \begin{cases} 1 & \text{if } i \not\in I \text{ and } i \leq c_n, \\
0 & \text{if } i \notin I \text{ and } i > c_n, \text{ or } i \in I \text{ and } i \leq c_n, \\
-1 & \text{if } i \in I \text{ and } i > c_n. \end{cases}
\]
The only possible values of $(d_{i_{2s-1}} - d_{i_{2s-1}}') - (d_{i_{2s}} - d_{i_{2s}}')$ are therefore $-1$, $0$, $1$ and $2$. If we have $i_{2s-1}, i_{2s} \leq c_n$ or $i_{2s-1}, i_{2s} > c_n$, then the value 2 is not possible and
\[
(d_{i_{2s-1}} - d_{i_{2s-1}}') - (d_{i_{2s}} - d_{i_{2s}}') = 1 \qquad \Leftrightarrow \qquad i_{2s-1} \not\in I \text{ and } i_{2s} \in I.
\]
Furthermore note that of the $2B$ pieces of horizontal boundary in column $n$, one is above row 1 (as $r_1 = n$, so $1 \in I$) and exactly $B-1$ are between a pair of cells with row indices $i$ and $i+1$, such that $i \not\in I$ and $i+1 \in I$. We now distinguish between four cases.

\textit{Case 1.} Suppose $i_{2t+1} \leq c_n$ and $i_{2t+1} \not\in I$. Then $2(d_{i_{2t+1}} - d_{i_{2t+1}}') = 2$. In the first $c_n$ cells of column $n$, there is at least one cell (the one with row index $i_{2t+1}$) that has value 0, hence $B \geq 2$ and there is a cell with row index greater than $i_{2t+1}$ with value 1. This means that there are at most $B-2$ pairs $(i_{2s-1}, i_{2s})$ such that $i_{2s-1} \not\in I$ and $i_{2s} \in I$. Also, $i_{2s-1}, i_{2s} \leq c_n$ for all $s$. So
\[
\sum_{s=1}^t \left( (d_{i_{2s-1}} - d_{i_{2s-1}}') - (d_{i_{2s}} - d_{i_{2s}}') \right) + 2(d_{i_{2t+1}} - d_{i_{2t+1}}') \leq (B-2) + 2 = B \leq 2B-2.
\]

\textit{Case 2.} Suppose $i_{2t+1} \leq c_n$ and $i_{2t+1} \in I$. Then $2(d_{i_{2t+1}} - d_{i_{2t+1}}') = 0$. Now there are at most $B-1$ pairs $(i_{2s-1}, i_{2s})$ such that $i_{2s-1} \not\in I$ and $i_{2s} \in I$. Also, $i_{2s-1}, i_{2s} \leq c_n$ for all $s$. So
\[
\sum_{s=1}^t \left( (d_{i_{2s-1}} - d_{i_{2s-1}}') - (d_{i_{2s}} - d_{i_{2s}}') \right) + 2(d_{i_{2t+1}} - d_{i_{2t+1}}') \leq B-1 \leq 2B-2.
\]

\textit{Case 3.} Suppose $i_{2t+1} > c_n$ and $B \geq 2$. Then $2(d_{i_{2t+1}} - d_{i_{2t+1}}') \leq 0$. Again there are at most $B-1$ pairs $(i_{2s-1}, i_{2s})$ such that $i_{2s-1} \not\in I$ and $i_{2s} \in I$. If there does not exist an $u$ such that $i_{2u-1} \leq c_n$ and $i_{2u} > c_n$, then we are done, as in the previous case. If there does exist such an $u$, then
\[
(d_{i_{2u-1}} - d_{i_{2u-1}}') - (d_{i_{2u}} - d_{i_{2u}}') = 2 \qquad \Leftrightarrow \qquad  i_{2u-1} \not\in I \text{ and } i_{2u} \in I.
\]
If $(d_{i_{2u-1}} - d_{i_{2u-1}}') - (d_{i_{2u}} - d_{i_{2u}}') = 2$, then on the right-hand side of (\ref{altgrenstoprove}) we have a 2 and at most $B-2$ times a 1. If not, then we have no 2 and at most $B$ times a 1. In both cases we find
\[
\sum_{s=1}^t \left( (d_{i_{2s-1}} - d_{i_{2s-1}}') - (d_{i_{2s}} - d_{i_{2s}}') \right) + 2(d_{i_{2t+1}} - d_{i_{2t+1}}') \leq B \leq 2B-2.
\]

\textit{Case 4.} Suppose $B=1$. Then $i \in I \Leftrightarrow i\leq c_n$, hence
\[
d_i' = d_i \qquad \text{for all } i.
\]
Therefore
\[
\sum_{s=1}^t \left( (d_{i_{2s-1}} - d_{i_{2s-1}}') - (d_{i_{2s}} - d_{i_{2s}}') \right) + 2(d_{i_{2t+1}} - d_{i_{2t+1}}') = 0 = 2B-2.
\]

In all possible cases we have now proved inequality (\ref{altgrenstoprove}), which finishes the proof of (\ref{altgrens1}).

Now we prove (\ref{altgrens2}). Let $F$ be a binary $m \times n$ image with row sums $\mathcal{R}$ and column sums $\mathcal{C}$. Define $\bar{F}$ as the binary $m \times n$ image that has zeroes where $F$ has ones and ones where $F$ has zeroes. Let $(\bar{r}_1, \ldots, \bar{r}_m)$ be the row sums of $\bar{F}$ and $(\bar{c}_1, \ldots, \bar{c}_n)$ the column sums. Define $\bar{b}_i = \# \{j: \bar{c}_j \geq i\}$ and $\bar{d}_i = \bar{b}_i - \bar{r}_{m+1-i}$ for $i = 1, 2, \ldots, m$. As $\bar{r}_i = n-r_i$ and $\bar{c}_j= m-c_j$ for all $i$ and $j$, we have
\[
\bar{b}_i = \# \{j : m-c_j \geq i\} = \# \{j : c_j \leq m-i\} = n - \# \{j: c_j \geq m+1-i\} = n-b_{m+1-i}.
\]
Hence
\[
\bar{d}_i = \bar{b}_i - \bar{r}_{m+1-i} = n-b_{m+1-i} - n + r_{m+1-i} = -d_{m+1-i}.
\]
As $\bar{r}_1 = 0$ and $\bar{r}_m = n$, we may apply (\ref{altgrens1}) to the row sums $(\bar{r}_m, \bar{r}_{m-1}, \ldots, \bar{r}_1)$. We write the subset of the row indices we use as $(m+1-i_{2t+1}, m+1-i_{2t}, \ldots, m+1-i_1)$ with $i_1 < i_2 < \ldots < i_{2t+1}$. We find that for the total length $\bar{L}_h$ of the horizontal boundary of $\bar{F}$ holds:
\begin{align*}
\bar{L}_h &\geq 2n + \bar{d}_{m+1-i_{2t+1}} - \bar{d}_{m+1-i_{2t}} + \bar{d}_{m+1-i_{2t-1}} - \cdots - \bar{d}_{m+1-i_2} + 2\bar{d}_{m+1-i_1} \\ &= 2n - d_{i_{2t+1}} + d_{i_{2t}} - d_{i_{2t-1}} + \cdots + d_{i_2} - 2d_{i_1}.
\end{align*}
In each column of $\bar{F}$, the number of horizontal pieces of boundary is equal to the number of pairs of neighbouring cells such that one cell has value 1 and the other has value 0, plus one for the boundary below row $m$. In each column of $F$, the number of horizontal pieces of boundary is equal to the number of pairs of neighbouring cells such that one cell has value 1 and the other has value 0, plus one for the boundary above row 1. As in each column the number of pairs of neighbouring cells such that one cell has value 1 and the other has value 0, is the same in $F$ and in $\bar{F}$, we have $\bar{L}_h = L_h$. Hence
\[
L_h \geq 2n - d_{i_{2t+1}} + d_{i_{2t}} - d_{i_{2t-1}} + \cdots + d_{i_2} - 2d_{i_1}.
\]

\end{proof}

\section{Some examples and a corollary}\label{examples}

To illustrate Theorem \ref{altgrens}, we apply it to two small examples.

\begin{example}\label{exaltgrens1}
Let $m=n=10$ and let row sums $(10, 7, 7, 5, 4, 3, 5, 6, 1, 0)$ and column sums $(8,8,8,8,6,3,2,2,2,1)$ be given. We compute $b_i$ and $d_i$, $i=1, 2, \ldots, 10$ as shown below.

\begin{tabular}{c|*{10}{c}}
$i$   & 1  & 2 & 3 & 4 & 5 & 6 & 7 & 8 & 9 & 10\\
\hline
$b_i$ & 10 & 9 & 6 & 5 & 5 & 5 & 4 & 4 & 0 & 0 \\
$r_i$ & 10 & 7 & 7 & 5 & 4 & 3 & 5 & 6 & 1 & 0 \\
\hline
$d_i$ & $0$ & $+2$ & $-1$ & $0$ & $+1$ & $+2$ & $-1$ & $-2$ & $-1$ & $0$
\end{tabular}

We take $t=1$, $i_1 = 2$, $i_2 = 3$ and $i_3 = 6$. Now (\ref{altgrens1}) tells us that
\[
L_h \geq 20 + 2 - (-1) + 2\cdot 2 = 27.
\]
Alternatively, we take $t=2$, $i_1=2$, $i_2=3$, $i_3=6$, $i_4=8$ and $i_5 =10$. Now \eqref{altgrens1} tells us that
\[
L_h \geq 20 + 2 - (-1) + 2 - (-2) + 2 \cdot 0 = 27.
\]
As $L_h$ must be even, we conclude $L_h \geq 28$. This bound is sharp: in Figure \ref{figaltgrens1} a binary image $F$ with the given row and column sums is shown, for which $L_h = 28$.
\end{example}

\begin{example}\label{exaltgrens2}
Let $m=n=10$ and let row sums $(10, 9, 7, 6, 8, 4, 5, 2, 3, 0)$ and column sums $(9, 8, 8, 6, 6, 4, 4, 4, 3, 2)$ be given. We compute $b_i$ and $d_i$, $i=1, 2, \ldots, 10$ as shown below.

\begin{tabular}{c|*{10}{c}}
$i$   & 1  & 2 & 3 & 4 & 5 & 6 & 7 & 8 & 9 & 10\\
\hline
$b_i$ & 10 & 10 & 9 & 8 & 5 & 5 & 3 & 3 & 1 & 0 \\
$r_i$ & 10 & 9  & 7 & 6 & 8 & 4 & 5 & 2 & 3 & 0 \\
\hline
$d_i$ & $0$ & $+1$ & $+2$ & $+2$ & $-3$ & $+1$ & $-2$ & $+1$ & $-2$ & $0$
\end{tabular}

We take $t=2$, $i_1 = 5$, $i_2 = 6$, $i_3 = 7$, $i_4=8$ and $i_5 = 9$. Now (\ref{altgrens2}) tells us that
\[
L_h \geq 20 - (-2) + 1 - (-2) + 1 - 2\cdot (-3) = 32.
\]
This bound is sharp: in Figure \ref{figaltgrens2} a binary image $F$ with the given row and column sums is shown, for which $L_h = 32$.
\end{example}

\begin{figure}
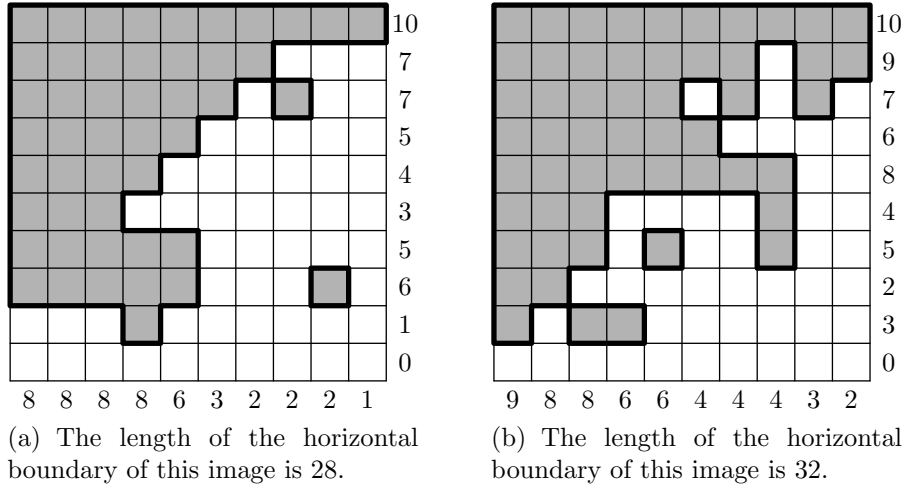

  \begin{center}
    \subfigure[The length of the horizontal boundary of this image is 28.]{\label{figaltgrens1}\includegraphics{plaatje.1}}
    \qquad
    \subfigure[The length of the horizontal boundary of this image is 32.]{\label{figaltgrens2}\includegraphics{plaatje.2}}
  \end{center}
\caption{The binary images from Examples \ref{exaltgrens1} and \ref{exaltgrens2}. The grey cells have value 1, the other cells value 0. The numbers indicate the row and column sums.}
\end{figure}

In the Introduction we mentioned two simple bounds of the length of the boundary. We recall them here, just for the horizontal boundary. The first one uses that in every column, there are at least two pieces of boundary, so if there are $n$ columns with nonzero sums, then
\begin{equation}\label{simple1}
L_h \geq 2n.
\end{equation}
The other bound computes the sum of the absolute differences between consecutive row sums, which yields
\begin{equation}\label{simple2}
L_h \geq r_1 + \sum_{i=1}^{m-1} |r_i - r_{i+1}| + r_m.
\end{equation}

In order to compare the bounds in Theorem \ref{altgrens} to these two simple bounds, we construct two families of examples.

\begin{example}\label{exaltgrens3}
Let the number of columns $n$ be even. Let $m=n+2$. Define line sums
\[
\mathcal{C} = (n, n, n-2, n-2, \ldots, 4, 4, 2, 2), \quad \mathcal{R} = (n, n-1, n-1, n-3, n-3, \ldots, 3, 3, 1, 1, 0).
\]
We calculate
\[
(b_1, b_2, \ldots, b_m) = (n, n, n-2, n-2, \ldots, 2, 2, 0, 0),
\]
\[
(d_1, d_2, \ldots, d_m) = (0, +1, -1, +1, -1, \ldots, +1, -1, +1, -1, 0).
\]
Now \eqref{altgrens1} tells us that
\[
L_h \geq 2n + \frac{n}{2} \cdot (1 - - 1) + 2\cdot 0 = 3n.
\]
On the other hand, \eqref{simple1} says $L_h \geq 2n$, while \eqref{simple2} gives
\[
L_h \geq n + 1 + \frac{n-2}{2} \cdot 2 + 1 = 2n.
\]
So Theorem \ref{altgrens} gives a much better bound in this family of examples. In fact, it is sharp: there exists a binary image with the length of the boundary equal to $3n$. Such an image is easy to construct; see for an example Figure \ref{figaltgrens3}.
\end{example}

\begin{example}\label{exaltgrens4}
Let $m=n+2$. Define line sums
\[
\mathcal{C} = (2, 2, 2, \ldots, 2, 2, 2), \quad \mathcal{R} = (n, 1, 1, 1, \ldots, 1, 1, 1, 0).
\]
We calculate
\[
(b_1, b_2, \ldots, b_m) = (n, n, 0, 0, 0, \ldots, 0, 0, 0),
\]
\[
(d_1, d_2, \ldots, d_m) = (0, +(n-1), -1, -1, -1, \ldots, -1, -1, -1, 0).
\]
Now \eqref{altgrens1} tells us that
\[
L_h \geq 2n + 2\cdot (n-1) = 4n-2.
\]
On the other hand, \eqref{simple1} says $L_h \geq 2n$, while \eqref{simple2} gives
\[
L_h \geq n + (n-1) + 1 = 2n.
\]
So again Theorem \ref{altgrens} gives a much better bound. In fact, it is sharp: there exists a binary image with the length of the boundary equal to $4n-2$. Such an image is easy to construct; see for an example Figure \ref{figaltgrens4}.
\end{example}

\begin{figure}
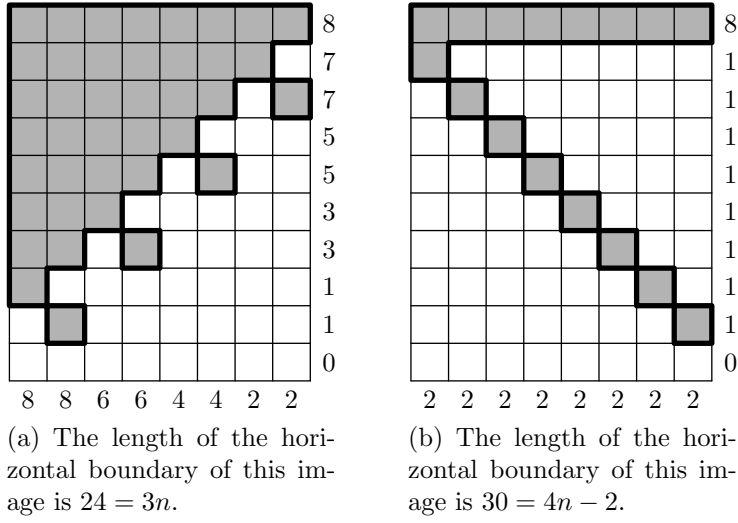

  \begin{center}
    \subfigure[The length of the horizontal boundary of this image is $24 =3n$.]{\label{figaltgrens3}\includegraphics{plaatje.15}}
    \qquad
    \subfigure[The length of the horizontal boundary of this image is $30=4n-2$.]{\label{figaltgrens4}\includegraphics{plaatje.16}}
  \end{center}
\caption{Binary images from Examples \ref{exaltgrens3} and \ref{exaltgrens4}, with $n=8$. The grey cells have value 1, the other cells value 0. The numbers indicate the row and column sums.}
\end{figure}

We can easily generalise the result from Theorem \ref{altgrens} to the case where the conditions $r_1 = n$ and $r_m = 0$ are not satisfied.

\begin{corollary}\label{corollaltgrens}
Let row sums $\mathcal{R} = (r_1, r_2, \ldots, r_m)$ and column sums $\mathcal{C} = (c_1, c_2, \ldots, c_n)$ be given. Let $L_h$ be the total length of the horizontal boundary of an image with line sums $(\mathcal{R}, \mathcal{C})$. Define $b_i = \#\{j: c_j \geq i\}$ and $d_i = b_i - r_i$ for $i = 1, 2, \ldots, m$. Also set $d_0 = d_{m+1} = 0$. For any integer $t \geq 0$ and any subset $\{i_1, i_2, \ldots, i_{2t+1} \} \subset \{0, 1, 2, \ldots, m, m+1\}$ with $i_1 < i_2 < \ldots < i_{2t+1}$ we have
\begin{align}
L_h &\geq 2r_1 + d_{i_1} - d_{i_2} + d_{i_3} - \cdots - d_{i_{2t}} + 2d_{i_{2t+1}}, \label{altgrens3} \\
L_h &\geq 2r_1 - d_{i_{2t+1}} + d_{i_{2t}} - d_{i_{2t-1}} + \cdots + d_{i_2} - 2d_{i_1} \label{altgrens4}.
\end{align}
\end{corollary}

\begin{proof}
Let $F$ be a binary image with line sums $(\mathcal{R},\mathcal{C})$ and a horizontal boundary of total length $L_h$. Construct $F'$ by adding a row above row 1 with row sum $n$ and a row below row $m$ with row sum $0$. Let $L_h'$ be the length of the horizontal boundary of $F'$. We have $L_h' = L_h + 2(n-r_1)$. The column sums of $F'$ are $c_j' = c_j + 1$, $j = 1, 2, \ldots, n$. The row sums are $r_1' = n$, $r_i' = r_{i-1}$ for $i=2, 3, \ldots, m+1$ and $r_{m+2}' = 0$. Let $b_i' = \#\{j: c_j' \geq i\}$ and $d_i' = b_i' - r_i'$ for $i = 1, 2, \ldots, m$. Then for all $i = 2, 3, \ldots, m+1$ we have
\[
b_i' = \#\{j: c_j + 1 \geq i\} = \#\{j: c_j \geq i-1\} = b_{i-1},
\]
so $d_i' = b_{i-1} - r_{i-1} = d_{i-1}$. Also, $d_1' = d_0 =0$ and $d_{m+2}' = d_{m+1}=0$. We apply Theorem \ref{altgrens} to $F'$ with the set of indices $\{i_1+1, i_2 +1, \ldots, i_{2t+1}+1\}$ and we find
\begin{align*}
L_h' &\geq 2n + d_{i_1+1}' - d_{i_2+1}' + d_{i_3+1}' - \cdots - d_{i_{2t}+1}' + 2d_{i_{2t+1}+1}' \\ &= 2n + d_{i_1} - d_{i_2} + d_{i_3} - \cdots - d_{i_{2t}} + 2d_{i_{2t+1}},\\
L_h' &\geq 2n - d_{i_{2t+1}+1}' + d_{i_{2t}+1}' - d_{i_{2t-1}+1}' + \cdots + d_{i_2+1}' - 2d_{i_1+1}' \\ &= 2n - d_{i_{2t+1}} + d_{i_{2t}} - d_{i_{2t-1}} + \cdots + d_{i_2} - 2d_{i_1},
\end{align*}
and therefore
\begin{align*}
L_h &\geq 2r_1 + d_{i_1} - d_{i_2} + d_{i_3} - \cdots - d_{i_{2t}} + 2d_{i_{2t+1}}, \\
L_h &\geq 2r_1 - d_{i_{2t+1}} + d_{i_{2t}} - d_{i_{2t-1}} + \cdots + d_{i_2} - 2d_{i_1}.
\end{align*}
\end{proof}

\section{A variation}\label{variation}

\begin{theorem}\label{thmsplits}
Let row sums $\mathcal{R} = (r_1, r_2, \ldots, r_m)$ and column sums $\mathcal{C} = (c_1, c_2, \ldots, c_n)$ be given, where $r_1=n$, $r_m = 0$. Suppose there exists an image $F$ with line sums $(\mathcal{R}, \mathcal{C})$ and let $L_h(F)$ be the total length of the horizontal boundary of this image. Define $b_i = \#\{j: c_j \geq i\}$ and $d_i = b_i - r_i$ for $i = 1, 2, \ldots, m$. Let $k$ be an integer with $2 \leq k \leq m-1$ such that $d_k < 0$ and $d_{k+1} \geq 0$. Let $\sigma = \sum_{i=1}^k d_k$. For any integers $t, s \geq 0$ and any sets $\{i_1, i_2, \ldots, i_{2t+1} \} \subset \{1, 2, \ldots, k-1, k, m\}$ with $i_1 < i_2 < \ldots < i_{2t+1}$ and $\{\tilde i_1, \tilde i_2, \ldots, \tilde i_{2s+1}\} \subset \{1, k+1, k+2, \ldots, m-1, m\}$ with $\tilde i_1 < \tilde i_2 < \ldots < \tilde i_{2s+1}$ we have
\begin{align}
L_h(F) \geq 2n &+ d_{i_1} - d_{i_2} + d_{i_3} - \cdots - d_{i_{2t}} + 2d_{i_{2t+1}} \notag \\
 &+ d_{\tilde i_1} - d_{\tilde i_2} + d_{\tilde i_3} - \cdots - d_{\tilde i_{2s}} + 2d_{\tilde i_{2s+1}} - \sigma. \label{splits}
\end{align}
\end{theorem}

\begin{proof}
We will prove the theorem by induction on $\sigma$. Note that by \eqref{ryserconsequence} we have $\sigma \geq 0$, since the line sums are consistent.

As we are only considering the horizontal boundary, we may for convenience assume that $c_1 \geq c_2 \geq \ldots \geq c_n$.

Suppose $\sigma = 0$. Then
\[
\sum_{i=1}^k r_i = \sum_{i=1}^k b_i = \sum_{i=1}^k \#\{j: c_j \geq i\} = \sum_{j \mid c_j \leq k} c_j + \sum_{j \mid c_j>k} k.
\]
So in any column $j$ with $c_j > k$ we must have $(i,j) \in F$ for $1 \leq i \leq k$, and in any column $j$ with $c_j \leq k$ we must have $(i,j) \not\in F$ for $k+1 \leq i \leq m$. This means that we can split the image $F$ into four smaller images, one of which contains only ones and one of which contains only zeroes. The other two parts we call $F_1$ and $F_2$ (see Figure \ref{figinductiebasis}). In order to have images with the first row filled with ones and the last row filled with zeroes, we glue row $m$ to $F_1$ and row 1 to $F_2$. More precisely, let $F_1$ consist of rows $1, 2, \ldots, k-1, k$ and $m$ of $F$ and the columns $j$ with $c_j \leq k$; let $F_2$ consist of rows $1$ and $k+1, k+2, \ldots, m-1, m$ of $F$ and the columns $j$ with $c_j > k$.

\begin{figure}
\begin{center}
\includegraphics{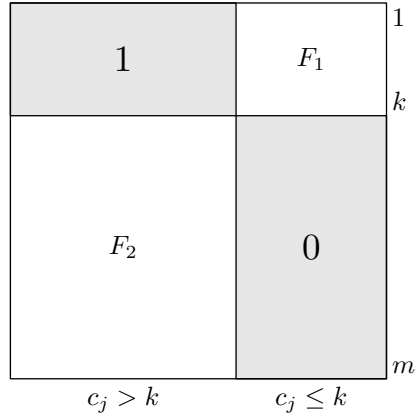}
\end{center}
\caption{Splitting the image $F$ into four smaller images.}
\label{figinductiebasis}
\end{figure}

The columns of $F$ with sum at most $k$ are exactly the columns with indices greater than $b_{k+1}$. Define $h = b_{k+1}$. Let $r_{1}^{(1)}$, $r_{2}^{(1)}, \ldots, r_k^{(1)}, r_{m}^{(1)}$ be the row sums of $F_1$, and let $r_{1}^{(2)}$, $r_{k+1}^{(2)}$, \ldots, $r_{m-1}^{(2)}$, $r_{m}^{(2)}$ be the row sums of $F_2$. We have
\[
r_i^{(1)} = r_i - h, \quad \text{for $1 \leq i \leq k$, and} \quad r_m^{(1)} = r_m,
\]
\[
r_i^{(2)} = r_i \quad \text{for $k+1 \leq i \leq m$, and} \quad r_1^{(2)} = h = r_1 - (n-h).
\]
Let $c_{h+1}^{(1)}$, $c_{h+2}^{(1)}$, \ldots, $c_{n-1}^{(1)}$, $c_n^{(1)}$ be the column sums of $F_1$, and let $c_1^{(2)}$, $c_2^{(2)}$, \ldots, $c_{h-1}^{(2)}$, $c_h^{(2)}$ be the column sums of $F_2$. We have
\[
c_j^{(1)} = c_j, \quad \text{and} \quad c_j^{(2)} = c_j - (k-1) \quad \text{for all $j$}.
\]
Define
\begin{align*}
b_1^{(1)} &= \# \{j \geq h+1 : c_j^{(1)} \geq 1 \}, & b_1^{(2)} &= \# \{j \leq h : c_j^{(2)} \geq 1 \},\\
b_2^{(1)} &= \# \{j \geq h+1 : c_j^{(1)} \geq 2 \}, & b_{k+1}^{(2)} &= \# \{j \leq h : c_j^{(2)} \geq 2 \},\\
&\vdots & &\vdots \\
b_k^{(1)} &= \# \{j \geq h+1 : c_j^{(1)} \geq k \}, & b_{m-1}^{(2)} &= \# \{j \leq h : c_j^{(2)} \geq m-k \},\\
b_m^{(1)} &= \# \{j \geq h+1 : c_j^{(1)} \geq k+1 \}, & b_m^{(2)} &= \# \{j \leq h : c_j^{(2)} \geq m-k+1 \}.\\
\end{align*}
For $1 \leq i \leq k$ we have
\[
b_i^{(1)} = \# \{j \geq h+1 : c_j^{(1)} \geq i \} = \# \{j \leq n : c_j \geq i\} - \# \{j \leq h : c_j \geq i\} = b_i - h.
\]
Also, $b_m^{(1)} = 0 = b_m$. For $k+1 \leq i \leq m$ we have
\[
b_{i}^{(2)} = \# \{j \leq h : c_j^{(2)} \geq i-k+1 \} = \# \{j \leq h : c_j \geq i \} \]\[= \# \{j \leq n : c_j \geq i\} - \#\{j \geq h+1 : c_j \geq i\} = b_i - 0 = b_i.
\]
Also, $b_1^{(2)} = h = b_1 - (n-h)$. Now define $d_i^{(1)} = b_i^{(1)} - r_i^{(1)}$ for $i \in \{1, 2, \ldots, k-1, k, m\}$ and $d_i^{(2)} = b_i^{(2)} - r_i^{(2)}$ for $i \in \{1, k+1, k+2, \ldots, m-1, m\}$. We find
\[
d_i^{(1)} = b_i - h - (r_i - h) = d_i, \quad \text{for $1 \leq i \leq k$,}
\]
\[
d_m^{(1)} = b_m - r_m = d_m,
\]
\[
d_i^{(2)} = b_i - r_i = d_i \quad \text{for $k+1 \leq i \leq m$}
\]
\[
d_1^{(2)} = b_1 - (n-h) - (r_1 - (n-h)) = d_1.
\]
All in all we conclude $d_i^{(1)} = d_i$ and $d_i^{(2)} = d_i$ for all $i$.

The total length of the horizontal boundary of $F$ in the columns $j$ with $c_j \leq k$ is exactly the same as the total length $L_h(F_1)$ of the horizontal boundary of $F_1$. The total length of the horizontal boundary of $F$ in the columns $j$ with $c_j > k$ is exactly the same as the total length $L_h(F_2)$ of the horizontal boundary of $F_2$. So $L_h(F) = L_h(F_1) + L_h(F_2)$. Note that $F_1$ has $n-b_{k+1}$ columns and $F_2$ has $b_{k+1}$ columns. By Theorem \ref{altgrens} applied to $F_1$ we know that for any integer $t \geq 0$ and any set $\{i_1, i_2, \ldots, i_{2t+1} \} \subset \{1, 2, \ldots, k-1, k, m\}$ with $i_1 < i_2 < \ldots < i_{2t+1}$ we have
\[
L_h(F_1) \geq 2(n-b_{k+1}) + d_{i_1} - d_{i_2} + d_{i_3} - \cdots - d_{i_{2t}} + 2d_{i_{2t+1}}.
\]
By the same theorem applied to $F_2$ we know that for any integer $t \geq 0$ and any set $\{\tilde i_1, \tilde i_2, \ldots, \tilde i_{2s+1}\} \subset \{1, k+1, k+2, \ldots, m-1, m\}$ with $\tilde i_1 < \tilde i_2 < \ldots < \tilde i_{2s+1}$ we have
\[
L_h(F_2) \geq 2b_{k+1} + d_{\tilde i_1} - d_{\tilde i_2} + d_{\tilde i_3} - \cdots - d_{\tilde i_{2s}} + 2d_{\tilde i_{2s+1}}.
\]
Adding these two results yields \eqref{splits}.

Now let $\sigma \geq 1$ and suppose that we have already proven the theorem for any image with $\sum_{i=1}^k d_i < \sigma$. Let
\begin{align*}
A_1 &= \max\{ d_{i_1} - d_{i_2} + d_{i_3} - \cdots - d_{i_{2t}} + 2d_{i_{2t+1}}\}, \\
A_2 &= \max\{ d_{\tilde i_1} - d_{\tilde i_2} + d_{\tilde i_3} - \cdots - d_{\tilde i_{2s}} + 2d_{\tilde i_{2s+1}} \},
\end{align*}
where the first maximum is taken over all integers $t \geq 0$ and sets $\{i_1, i_2, \ldots, i_{2t+1} \} \subset \{1, 2, \ldots, k-1, k, m\}$ with $i_1 < i_2 < \ldots < i_{2t+1}$, and the second maximum over all integers $s \geq 0$ and sets $\{\tilde i_1, \tilde i_2, \ldots, \tilde i_{2s+1}\} \subset \{1, k+1, k+2, \ldots, m-1, m\}$ with $\tilde i_1 < \tilde i_2 < \ldots < \tilde i_{2s+1}$. Furthermore, fix $i_1,  i_2, \ldots, i_{2t+1}$ and $\tilde i_1, \tilde i_2, \ldots, \tilde i_{2s+1}$ such that these maxima are attained.

Since $d_k<0$ by definition of $k$, and since $d_m = 0$, we have
\[
d_{i_1} - d_{i_2} + d_{i_3} - \cdots - d_{i_{2t}} + 2d_{k} < d_{i_1} - d_{i_2} + d_{i_3} - \cdots - d_{i_{2t}} + 2d_{m}.
\]
If $i_{2t+1} = k$, this would contradict the maximality of $A_1$, so we conclude
\begin{equation}\label{eqfact}
i_{2t+1} \neq k.
\end{equation}

We also know $d_{k+1} \geq 0$ by definition of $k$, and $d_1 = 0$. So if $s \geq 1$, then
\[
d_{1} - d_{k+1} + d_{\tilde i_3} - \cdots - d_{\tilde i_{2s}} + 2d_{\tilde i_{2s+1}} \leq d_{\tilde i_3} - \cdots - d_{\tilde i_{2s}} + 2d_{\tilde i_{2s+1}}.
\]
This means that if $s \geq 1$, we may assume without loss of generality that $(\tilde i_1, \tilde i_2) \neq (1, k+1)$. Also,
\[
d_{1} - d_{\tilde i_2} + d_{\tilde i_3} - \cdots - d_{\tilde i_{2s}} + 2d_{\tilde i_{2s+1}} \leq d_{k+1} - d_{\tilde i_2} + d_{\tilde i_3} - \cdots - d_{\tilde i_{2s}} + 2d_{\tilde i_{2s+1}}.
\]
This means that if $s \geq 1$ and $\tilde i_2 > k+1$, we may assume that $\tilde i_1 \neq 1$.
Finally,
\[
2 d_1 \leq 2d_{k+1},
\]
so if $s=1$ we may also assume that $\tilde i_1 \neq 1$.

All in all we may assume in all cases that
\begin{equation}\label{eqassume}
\tilde i_1 \neq 1.
\end{equation}

It suffices to prove
\begin{equation}\label{tebewijzen}
L_h(F) \geq 2n + A_1 + A_2 - \sigma.
\end{equation}

Let $j$ with $1 \leq j \leq n$ be such that $\# \big( \{(1,j), (2,j), \ldots, (k,j) \} \cap F \big) < \min(c_j, k)$, i.e. in column $j$ there is at least one one in rows $k+1, k+2, \ldots, m$ and at least one zero in rows $1, 2, \ldots, k$. Such a column exists, because
\[
\sum_{i=1}^k r_i < \sum_{i=1}^k b_i = \sum_{i=1}^k \#\{j: c_j \geq i\} = \sum_{j \mid c_j \leq k} c_j + \sum_{j \mid c_j>k} k.
\]

We will now consider various cases.

\textit{Case} 1. Suppose that there exist integers $l \geq 2$, $h \geq k+1$ and $u \geq 0$ such that $l+u \leq k$, $h+u \leq m-1$ and
\begin{itemize}
\item $(l-1, j) \in F$, and
\item $(l, j), (l+1, j), \ldots, (l+u, j) \not\in F$, and
\item $(h,j), (h+1, j), \ldots, (h+u,j) \in F$, and
\item $(h+u+1,j) \not\in F$, and
\item $(l+u+1, j) \in F$ or $(h-1,j) \not\in F$.
\end{itemize}

\begin{figure}
\begin{center}
\includegraphics{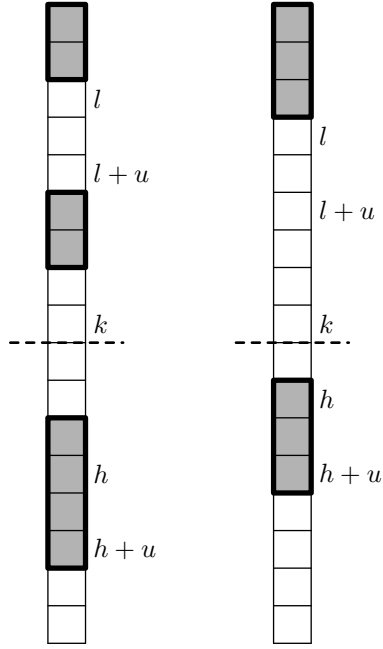}
\end{center}
\caption{Two possibilities for column $j$ in Case 1. The grey cells have value 1, the other cells value 0.}
\label{figcase1before}
\end{figure}

We define a new image $F'$ by moving the ones at $(h,j), (h+1,j), \ldots, (h+u,j)$ to $(l,j), (l+1,j), \ldots, (l+u,j)$; that is,
\[
F' = F \cup \{ (l,j), (l+1,j), \ldots, (l+u,j) \} \backslash \{ (h,j), (h+1,j), \ldots, (h+u,j) \}.
\]
The column sums of $F'$ are identical to the column sums of $F$. The row sums $r_i'$ of $F'$ are given by
\[
r_i' = \begin{cases} r_i + 1 & \text{if $l \leq i \leq l+u$}, \\ r_i - 1 & \text{if $h \leq i \leq h+u$}, \\ r_i & \text{else.} \end{cases}
\]
Define $d_i' = b_i - r_i'$ and $\sigma' = \sum_{i=1}^k d_i' = \sigma - (u+1)$. By the induction hypothesis, we have for the total length $L_h(F')$ of the horizontal boundary of $F'$
\[
L_h(F') \geq 2n + A_1' + A_2' - \sigma',
\]
where
\begin{align*}
A_1' &= d_{i_1}' - d_{i_2}' + d_{i_3}' - \cdots - d_{i_{2t}}' + 2d_{i_{2t+1}}',\\
A_2' &= d_{\tilde i_1}' - d_{\tilde i_2}' + d_{\tilde i_3}' - \cdots - d_{\tilde i_{2s}}' + 2d_{\tilde i_{2s+1}}'.
\end{align*}
By moving the $u+1$ ones in column $j$, the piece of horizontal boundary between row $l-1$ and row $l$ has vanished, just like the piece of horizontal boundary between row $h+u$ and $h+u+1$. If $(l+u+1,j) \in F$, the piece of horizontal boundary between row $l+u$ and row $l+u+1$ has also vanished, but there may be a new piece of horizontal boundary between row $h-1$ and $h$. On the other hand, if $(h-1,j) \not\in F$, the piece of horizontal boundary between row $h-1$ and row $h$ has vanished, but there may be a new piece of horizontal boundary between row $l+u$ and $l+u+1$. At least one of both is the case. All in all, we have $L_h(F') \leq L_h(F) - 2$.

\begin{figure}
\begin{center}
\includegraphics{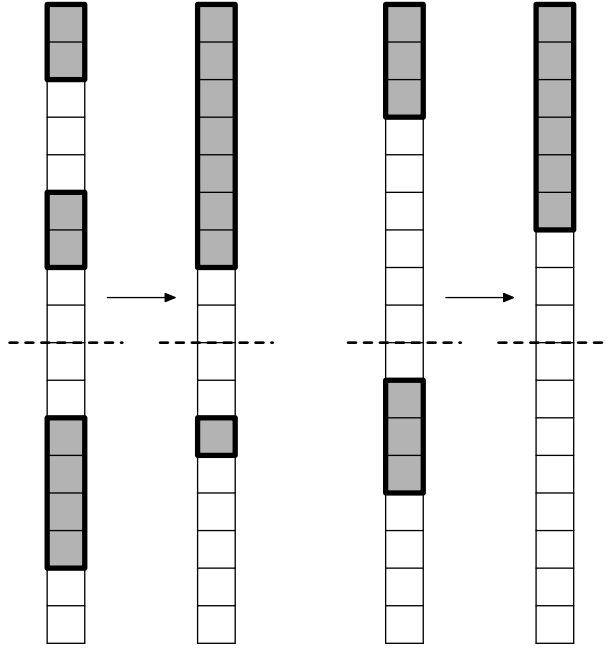}
\end{center}
\caption{Moving ones in Case 1, in both possible configurations. The grey cells have value 1, the other cells value 0.}
\label{figcase1after}
\end{figure}

Furthermore, some of the $d_i'$ involved in $A_1'$ or $A_2'$ may be different from the corresponding $d_i$. Since $\{i_1, i_2, \ldots, i_{2t+1} \} \subset \{1, 2, \ldots, k-1, k, m\}$, we have $d_i' = d_i$ or $d_i' = d_i-1$ for $i \in \{i_1, i_2, \ldots, i_{2t+1} \}$. The values of $i$ for which $d_i' = d_i-1$, are all consecutive. Since the coefficients for $d_i$ in $A_1$ are alternatingly positive and negative, and there is only one positive coefficient that is $+2$ rather than $+1$, we have
\[
A_1' = d_{i_1}' - d_{i_2}' + d_{i_3}' - \cdots - d_{i_{2t}}' + 2d_{i_{2t+1}}' \geq d_{i_1} - d_{i_2} + d_{i_3} - \cdots - d_{i_{2t}} + 2d_{i_{2t+1}} - 2 = A_1 - 2.
\]
Since $\{\tilde i_1, \tilde i_2, \ldots, \tilde i_{2s+1}\} \subset \{1, k+1, k+2, \ldots, m-1, m\}$, we have $d_i' = d_i$ or $d_i' = d_i + 1$ for $i \in \{\tilde i_1, \tilde i_2, \ldots, \tilde i_{2s+1}\}.$ By a similar argument as above and by the fact that all negative coefficients in $A_2$ are equal to $-1$, we have
\[
A_2' \geq A_2 - 1.
\]
Finally, we have $\sigma' = \sigma - (u+1) \leq \sigma - 1$. We conclude
\begin{align*}
L_h(F) &\geq L_h(F') + 2 \\ &\geq 2n + A_1' + A_2' - \sigma' + 2 \\ &\geq 2n + (A_1 -2) + (A_2 - 1) - (\sigma-1) +2 \\ &= 2n + A_1 + A_2 - \sigma.
\end{align*}
This proves \eqref{tebewijzen} in Case 1.

\textit{Case} 2. Suppose that the conditions of Case 1 do not hold and furthermore that $(k,j) \in F$ and $(k+1,j) \in F$. Then there exist integers $l \geq 2$, $h \leq k$ and $u \geq 0$ such that $h \geq l+1$, $k+1 \leq h+u \leq m-1$ and
\begin{itemize}
\item $(l-1, j) \in F$, and
\item $(l, j), (l+1, j), \ldots, (h-1, j) \not\in F$, and
\item $(h,j), (h+1, j), \ldots, (h+u,j) \in F$, and
\item $(h+u+1,j) \not\in F$.
\end{itemize}
As Case 1 does not apply, we cannot change all zeroes in $(l, j)$, $(l+1,j)$, \ldots, $(h-1,j)$ into ones by moving ones from $(k+1,j)$, $(k+2,j)$, \ldots, $(h+u,j)$. This implies that $h-l > (h+u)-k \geq 1$, so $l < h-1$. We will now distinguish between several cases.

\begin{figure}
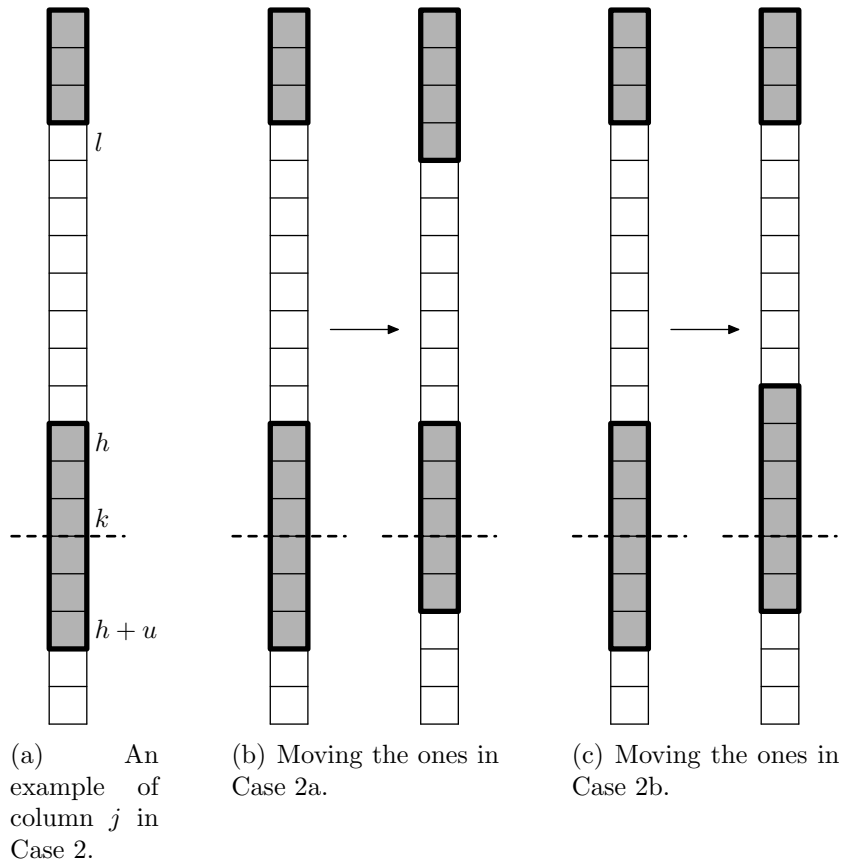

  \begin{center}
    \subfigure[An example of column $j$ in Case 2.]{\label{figcase2before}\includegraphics{plaatje.5}}
    \qquad
    \subfigure[Moving the ones in Case 2a.]{\label{figcase2a}\includegraphics{plaatje.8}}
    \qquad
    \subfigure[Moving the ones in Case 2b.]{\label{figcase2b}\includegraphics{plaatje.9}}
  \end{center}
\caption{Illustrations for Case 2 of the proof. The grey cells have value 1, the other cells value 0.}
\end{figure}

\textit{Case} 2a. Suppose that there does not exist an integer $r$ with $0 \leq r \leq t$ such that $l = i_{2r+1}$. We define a new image $F'$ by moving the one at $(h+u,j)$ to $(l,j)$; that is,
\[
F' = F \cup \{(l,j)\} \backslash \{(h+u,j)\}.
\]
We define $r_i'$, $d_i'$, $\sigma'$, $A_1'$, $A_2'$ and $L_h(F')$ similarly as in Case 1. As in Case 1 we have $A_2' \geq A_2 -1$. However, of the $d_i$ with $i \in \{1, 2, \ldots, k-1, k, m\}$ only one has changed (namely $d_l' = d_l - 1$), and we know that $d_l$ does not have a positive coefficient in $A_1$. So $A_1' \geq A_1$. Furthermore, $L_h(F') = L_h(F)$ and $\sigma' = \sigma - 1$. By applying the induction hypothesis to $F'$, we find
\begin{align*}
L_h(F) &= L_h(F') \\ &\geq 2n + A_1' + A_2' - \sigma' \\ &\geq 2n + A_1 + (A_2 - 1) - (\sigma-1) \\ &= 2n + A_1 + A_2 - \sigma.
\end{align*}
This proves \eqref{tebewijzen} in Case 2a.

\textit{Case} 2b. Suppose that there does not exist an integer $r$ with $0 \leq r \leq t$ such that $h-1 = i_{2r+1}$. We define a new image $F'$ by moving the one at $(h+u,j)$ to $(h-1,j)$; the rest of the proof is the same as in Case 2a.

\textit{Case} 2c. Suppose neither Case 2a nor Case 2b applies. Then there are integers $r_1$ and $r_2$ with $0 \leq r_1 < r_2 \leq t$ such that $l=i_{2r_1+1}$ and $h-1=i_{2r_2+1}$. Note that $r_1 < t$, so $d_l$ has coefficient $+1$ in $A_1$. Now let $v = i_{2r_1+2} < h-1$. Again, we distinguish between two cases.

\begin{figure}
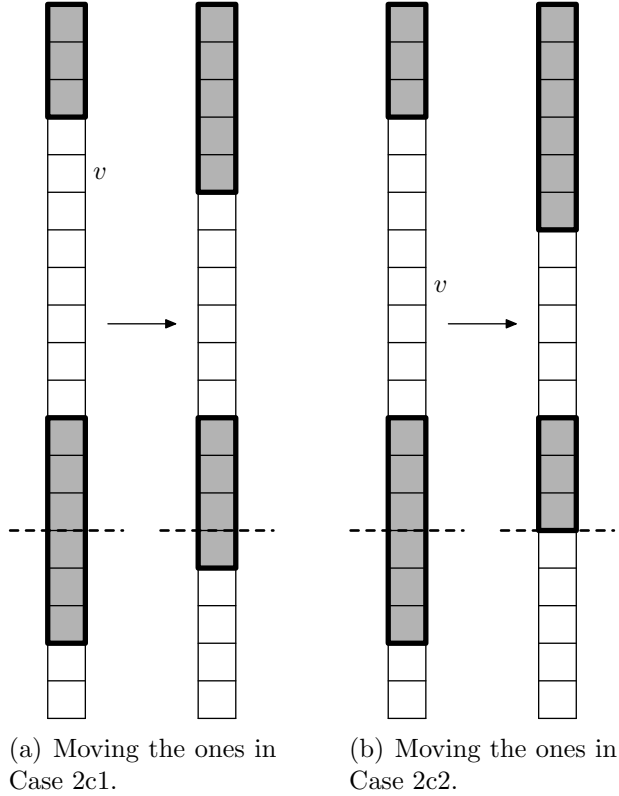

  \begin{center}
    \subfigure[Moving the ones in Case 2c1.]{\label{figcase2c1}\includegraphics{plaatje.10}}
    \qquad
    \subfigure[Moving the ones in Case 2c2.]{\label{figcase2c2}\includegraphics{plaatje.11}}
  \end{center}
\caption{More illustrations for Case 2 of the proof. The grey cells have value 1, the other cells value 0.}
\end{figure}

\textit{Case} 2c1. Suppose that $k+1 \leq h+u-v+l$. Then we define a new image $F'$ by moving the ones at $(h+u-v+l, j)$, $(h+u-v+l+1,j)$, \ldots, $(h+u,j)$ to $(l,j)$, $(l+1,j)$, \ldots, $(v,j)$; that is,
\[
F' = F \cup \{(l,j), (l+1,j), \ldots, (v,j)\} \backslash \{(h+u-v+l, j), (h+u-v+l+1,j), \ldots, (h+u,j) \}.
\]
We define $r_i'$, $d_i'$, $\sigma'$, $A_1'$, $A_2'$ and $L_h(F')$ similarly as in Case 1. As in Case 2a we have $A_2' \geq A_2 -1$ and $L_h(F') = L_h(F)$. Also, $\sigma' \leq \sigma - 1$. Furthermore, of the $d_i$ with $i \in \{1, 2, \ldots, k-1, k, m\}$ exactly two have changed: $d_l' = d_l -1$ and $d_v' = d_v - 1$. As $d_l$ has coefficient $+1$ in $A_1$ and $d_v$ has coefficient $-1$ in $A_1$, we have $A_1' = A_1$. By applying the induction hypothesis to $F'$, we find
\begin{align*}
L_h(F) &= L_h(F') \\ &\geq 2n + A_1' + A_2' - \sigma' \\ &\geq 2n + A_1 + (A_2 - 1) - (\sigma-1) \\ &= 2n + A_1 + A_2 - \sigma.
\end{align*}
This proves \eqref{tebewijzen} in Case 2c1.

\textit{Case} 2c2. Suppose that $k+1 > h+u-v+l$. Then we define a new image $F'$ by moving the ones at $(k+1, j)$, $(k+2,j)$, \ldots, $(h+u,j)$ to $(l,j)$, $(l+1,j)$, \ldots, $(l+h+u-k-1,j)$; that is,
\[
F' = F \cup \{(l,j), (l+1,j), \ldots, (l+h+u-k-1,j)\} \backslash \{(k+1, j), (k+2,j), \ldots, (h+u,j) \}.
\]
We define $r_i'$, $d_i'$, $\sigma'$, $A_1'$, $A_2'$ and $L_h(F')$ similarly as in Case 1. As in Case 2c1 we have $L_h(F') = L_h(F)$ and $\sigma' \leq \sigma - 1$. Since $l+h+u-k-1 < v$, of the $d_i$ with $i \in \{1, 2, \ldots, k-1, k, m\}$ exactly one has changed: $d_l' = d_l -1$. As $d_l$ has coefficient $+1$ in $A_1$, we have $A_1' = A_1-1$.

Now we consider $A_2'$. Some of the $d_i$ with $i \in \{\tilde i_1, \tilde i_2, \ldots, \tilde i_{2s+1}\}$ may have increased by 1. If $\tilde i_1 > h+u$, none of the row indices $k+1$, $k+2$, \ldots, $h+u$ occurs in $\{\tilde i_1, \tilde i_2, \ldots, \tilde i_{2s+1}\}$, and we have $A_2' = A_2$. If not, then $k+1 \leq \tilde i_1 \leq h+u$ (using \eqref{eqassume}). The values of $i$ for which $d_i' = d_i+1$, are all consecutive. Since the coefficients for $d_i$ in $A_1$ are alternatingly positive and negative, and since $\tilde i_1$ (which has a positive coefficient in $A_1$) is included in $\{k+1, k+2, \ldots, h+u\}$, we have $A_2' \geq A_2$.

By applying the induction hypothesis to $F'$, we find
\begin{align*}
L_h(F) &= L_h(F') \\ &\geq 2n + A_1' + A_2' - \sigma' \\ &\geq 2n + (A_1-1) + A_2 - (\sigma-1) \\ &= 2n + A_1 + A_2 - \sigma.
\end{align*}
This proves \eqref{tebewijzen} in Case 2c2, which completes the proof of Case 2.

\textit{Case} 3. Suppose that the conditions of Case 1 and Case 2 do not hold. By definition of $j$ we know that in column $j$ there is at least one one in rows $k+1$, $k+2$, \ldots, $m$. As Case 2 does not apply, we have $(k,j) \notin F$ or $(k+1,j) \not\in F$. If $(k,j) \in F$ (so $(k+1, j) \not\in F$) we can apply Case 1: let $l$ be the smallest integer such that $(l,j)\not\in F$, let $h'$ be the greatest integer such that $(h',j) \in F$, and let $u$ be maximal such that $(i,j) \not\in F$ for $l \leq i \leq l+u$ and $(i,j) \in F$ for $h'-u \leq i \leq h'$. Define $h = h'-u$. Since $(k,j) \in F$ and $(k+1,j) \not\in F$, we have $l+u < k$ and $h > k+1$, so all conditions of Case 1 are satisfied.

Hence we have $(k,j) \not\in F$. Now there exist integers $h \geq k+1$ and $u \geq 0$ such that $h+u \leq m-1$ and
\begin{itemize}
\item $(h-1,j) \not\in F$, and
\item $(i,j) \in F$ for $h \leq i \leq h+u$, and
\item $(h+u+1,j) \not\in F$.
\end{itemize}
Furthermore, let $l \leq k$ be such that $(l-1,j)\in F$ and $(l,j) \not\in F$. Since Case 1 does not apply, there does not exist an integer $u'$ such that $l+u' \leq k$, $(i,j) \not\in F$ for $l \leq i \leq l+u'$ and $(l+u'+1, j) \in F$. This means that $(i,j) \not\in F$ for all $i$ with $l \leq i \leq k+1$. Also, we could still apply Case 1 if there are at least as many zeroes in $(l,j)$, $(l+1,j)$, \ldots $(k,j)$ as there are ones in $(h,j)$, $(h+1, j)$, \ldots, $(h+u,j)$. Hence we must have $u+1 > k-l+1$.

\begin{figure}
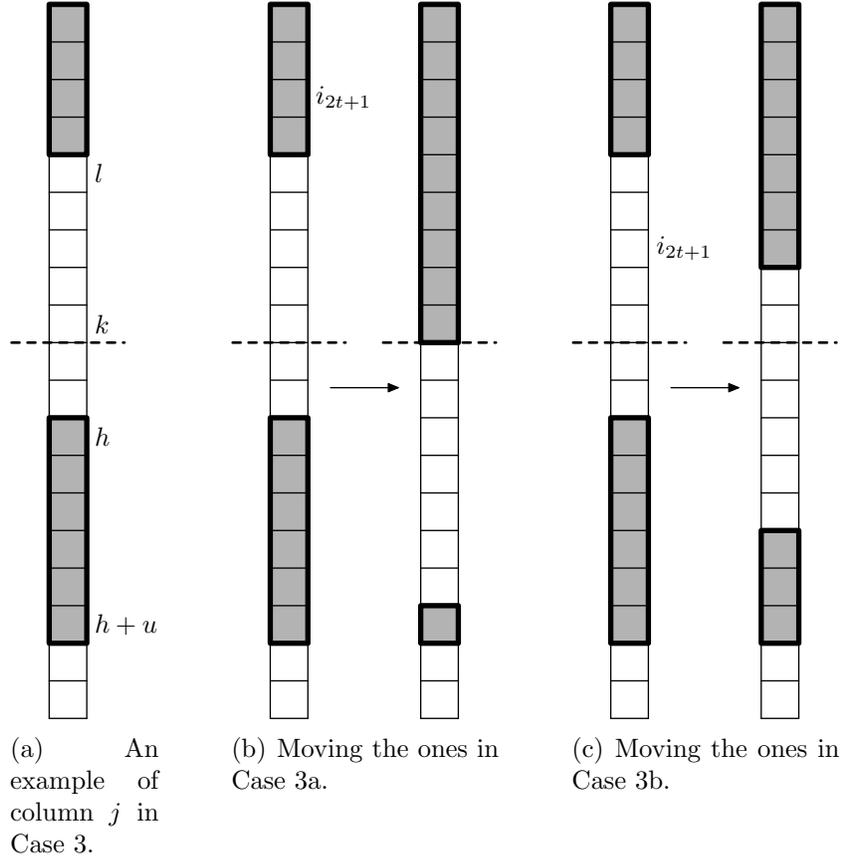

  \begin{center}
    \subfigure[An example of column $j$ in Case 3.]{\label{figcase3before}\includegraphics{plaatje.6}}
    \qquad
    \subfigure[Moving the ones in Case 3a.]{\label{figcase3a}\includegraphics{plaatje.12}}
    \qquad
    \subfigure[Moving the ones in Case 3b.]{\label{figcase3b}\includegraphics{plaatje.13}}
  \end{center}
\caption{Illustrations for Case 3 of the proof. The grey cells have value 1, the other cells value 0.}
\end{figure}

We will distinguish between various cases.

\textit{Case} 3a. Suppose that either $i_{2t+1} < l$ or $i_{2t+1} = m$. This means that none of the $d_i$ with $l \leq i \leq k$ has coefficient $+2$ in $A_1$. Since $u+1 > k-l+1$, we have $h+k-l < h+u$, so there are ones at $(h,j)$, $(h+1,j)$, \ldots, $(h+k-l,j)$. We define a new image $F'$ by moving those ones to $(l,j)$, $(l+1,j)$, \ldots, $(k,j)$; that is
\[
F' = F \cup \{(l,j), (l+1,j), \ldots, (k,j)\} \backslash \{(h,j), (h+1,j), \ldots, (h+k-l,j)\}.
\]
We define $r_i'$, $d_i'$, $\sigma'$, $A_1'$, $A_2'$ and $L_h(F')$ similarly as in Case 1. As in Case 1 we have $A_2' \geq A_2-1$. Furthermore, $L_h(F') = L_h(F)$.

Suppose $l=k$. Then only one $d_i$ with $i \in \{1, 2, \ldots, k-1, k, m\}$ has changed, namely $d_k' = d_k - 1$. We know that $d_k$ does not have a positive coefficient in $A_1$, since $k \neq i_{2t+1}$ (see \eqref{eqfact}) and $i_{2t-1} \leq k-1$. So $A_1' \geq A_1$. Also, $\sigma' = \sigma - 1$, so by applying the induction hypothesis to $F'$, we find
\begin{align*}
L_h(F) &= L_h(F') \\ &\geq 2n + A_1' + A_2' - \sigma' \\ &\geq 2n + A_1 + (A_2-1) - (\sigma -1) \\ &= 2n + A_1 + A_2 - \sigma.
\end{align*}

Now suppose that $l<k$. Then we have $\sigma' \leq \sigma - 2$. Furthermore, none of the $d_i$ with $l \leq i \leq k$ has coefficient $+2$ in $A_1$, so $A_1' \geq A_1-1$. By applying the induction hypothesis to $F'$, we find
\begin{align*}
L_h(F) &= L_h(F') \\ &\geq 2n + A_1' + A_2' - \sigma' \\ &\geq 2n + (A_1-1) + (A_2-1) - (\sigma -2) \\ &= 2n + A_1 + A_2 - \sigma.
\end{align*}
This proves \eqref{tebewijzen} in Case 3a.

\textit{Case} 3b. Suppose that $i_{2t+1} \geq l$, $i_{2t+1} \neq m$ and $i_{2t+1} \neq k-1$. Using \eqref{eqfact}, we then have $l \leq i_{2t+1} \leq k-2$. Since $u+1 > k-l+1$, we find that $u \geq k-l+1 \geq (l+2) -l +1 \geq 3$. We define a new image $F'$ by moving the ones at $(h,j)$, $(h+1,j)$ and $(h+2,j)$ to $(l,j)$, $(l+1,j)$ and $(l+2,j)$; that is,
\[
F' = F \cup \{(l,j), (l+1,j), (l+2,j)\} \backslash \{(h,j), (h+1,j), (h+2,j)\}.
\]
We define $r_i'$, $d_i'$, $\sigma'$, $A_1'$, $A_2'$ and $L_h(F')$ similarly as in Case 1. As in Case 1, we have $A_1' \geq A_1 -2$ and $A_2' \geq A_2 - 1$. Furthermore, $L_h(F') = L_h(F)$ and $\sigma' = \sigma-3$. By applying the induction hypothesis to $F'$, we find
\begin{align*}
L_h(F) &= L_h(F') \\ &\geq 2n + A_1' + A_2' - \sigma' \\ &\geq 2n + (A_1-2) + (A_2-1) - (\sigma -3) \\ &= 2n + A_1 + A_2 - \sigma.
\end{align*}
This proves \eqref{tebewijzen} in Case 3b.

\textit{Case} 3c. Suppose that neither Case 3a nor Case 3b applies. Then we have $i_{2t+1} = k-1$. Using \eqref{eqassume}, this means that $\tilde i_1 \geq k+1 > k-1 = i_{2t+1}$. We now apply Theorem \ref{altgrens} to the image $F$ and the row indices $\{i_1, i_2, \ldots, i_{2t}, k-1, k, \tilde i_1, \tilde i_2, \ldots, \tilde i_{2s+1}\}$:
\begin{align*}
L_h(F) &\geq 2n + d_{i_1} - d_{i_2} + \cdots - d_{i_{2t}} + d_{k-1} - d_k + d_{\tilde i_1} - d_{\tilde i_2} + \cdots - d_{\tilde i_{2s}} + 2 d_{\tilde i_{2s+1}} \\
&= 2n + A_1 - d_{k-1} - d_k + A_2.
\end{align*}

By Ryser's Theorem \cite{ryser} we have $\sum_{i=1}^{k-2} d_i \geq 0$, since the line sums are consistent, so
\[
\sigma = \sum_{i=1}^k d_i = \sum_{i=1}^{k-2} d_i + d_{k-1} + d_k \geq d_{k-1} + d_k.
\]
Hence
\[
L_h(F) \geq 2n+A_1 - d_{k-1} - d_k + A_2 \geq 2n + A_1 + A_2 - \sigma,
\]
which proves \eqref{tebewijzen} in Case 3c.

This finishes the proof of the theorem.
\end{proof}

\begin{example}\label{exsplitsing}
Let $m=n=12$ and let row sums $(12, 8, 9, 8, 8, 5, 5, 2, 3, 2, 1, 0)$ and column sums $(10, 8, 8, 8, 6, 6, 6, 3, 2, 2, 2, 2)$ be given. We compute $b_i$ and $d_i$, $i=1, 2, \ldots, 12$ as shown below.

\begin{tabular}{c|*{12}{c}}
$i$   & 1  & 2 & 3 & 4 & 5 & 6 & 7 & 8 & 9 & 10 & 11 & 12\\
\hline
$b_i$ & 12 & 12 & 8 & 7 & 7 & 7 & 4 & 4 & 1 & 1 & 0 & 0 \\
$r_i$ & 12 & 8 & 9 & 8 & 8 & 5 & 5 & 2 & 3 & 2 & 1 & 0 \\
\hline
$d_i$ & $0$ & $+4$ & $-1$ & $-1$ & $-1$ & $+2$ & $-1$ & $+2$ & $-2$ & $-1$ & $-1$ & $0$
\end{tabular}

Here \eqref{altgrens1} yields at most
\[
L_h \geq 24 + 4 - (-1) + 2 - (-1) + 2 \cdot 2 = 36,
\]
and \eqref{altgrens2} yields at most
\[
L_h \geq 24 - (-2) + 2 - (-1) + 2 - (-1) + 4 - 2 \cdot 0 = 36.
\]
However, we can apply Theorem \ref{thmsplits} with $k=5$ (note that $d_5 < 0$ and $d_6 \geq 0$). We have $\sigma = 1$. If we take $t=0$, $s=0$, $i_1 = 2$, $\tilde i_1 = 6$, $\tilde i_2 = 7$ and $\tilde i_3 = 8$, then we find
\[
L_h \geq 24 + 2 \cdot 4 + 2 - (-1) + 2\cdot 2 - 1 = 38.
\]
So in this example, Theorem \ref{thmsplits} gives a better bound than Theorem \ref{altgrens}. In fact, the bound of Theorem \ref{thmsplits} is sharp in this example: in Figure \ref{figexsplitsing} a binary image $F$ with the given row and column sums is shown, for which $L_h = 38$.
\end{example}

\begin{figure}
\begin{center}
\includegraphics{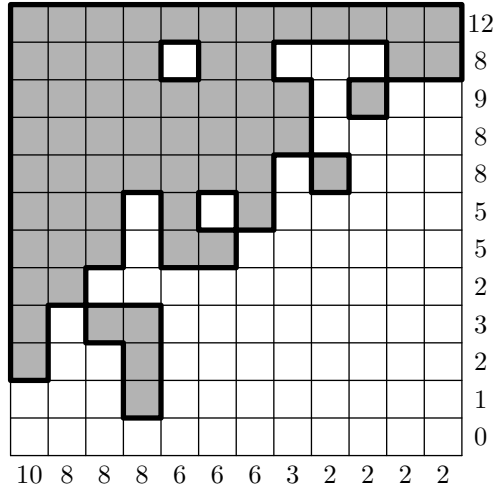}
\end{center}
\caption{The binary image from Examples \ref{exsplitsing}. The grey cells have value 1, the other cells value 0. The numbers indicate the row and column sums. The length of the horizontal boundary of this image is 38.}
\label{figexsplitsing}
\end{figure}

\begin{corollary}\label{corollsplits}
Let row sums $\mathcal{R} = (r_1, r_2, \ldots, r_m)$ and column sums $\mathcal{C} = (c_1, c_2, \ldots, c_n)$ be given. Suppose there exists an image $F$ with line sums $(\mathcal{R}, \mathcal{C})$ and let $L_h(F)$ be the total length of the horizontal boundary of this image. Define $b_i = \#\{j: c_j \geq i\}$ and $d_i = b_i - r_i$ for $i = 1, 2, \ldots, m$. Also set $d_0 = d_{m+1} = 0$. Let $k$ be an integer with $1 \leq k \leq m$ such that $d_k < 0$ and $d_{k+1} \geq 0$. Let $\sigma = \sum_{i=1}^k d_k$. For any integers $t, s \geq 0$ and any sets $\{i_1, i_2, \ldots, i_{2t+1} \} \subset \{0, 1, \ldots, k-1, k, m+1\}$ with $i_1 < i_2 < \ldots < i_{2t+1}$ and $\{\tilde i_1, \tilde i_2, \ldots, \tilde i_{2s+1}\} \subset \{0, k+1, k+2, \ldots, m, m+1\}$ with $\tilde i_1 < \tilde i_2 < \ldots < \tilde i_{2s+1}$ we have
\begin{align}
L_h(F) \geq 2r_1 &+ d_{i_1} - d_{i_2} + d_{i_3} - \cdots - d_{i_{2t}} + 2d_{i_{2t+1}} \notag \\
 &+ d_{\tilde i_1} - d_{\tilde i_2} + d_{\tilde i_3} - \cdots - d_{\tilde i_{2s}} + 2d_{\tilde i_{2s+1}} - \sigma.
\end{align}
\end{corollary}

\begin{proof}
Completely analogous to the proof of Corollary \ref{corollaltgrens}.
\end{proof}

\end{document}